\documentclass[11pt,a4paper]{amsart}
\usepackage{color}
\usepackage[T1]{fontenc}
\usepackage[utf8]{inputenc}
\usepackage{hyperref,amssymb,url,upref,verbatim,xspace,mathrsfs}
\usepackage{amsthm,amssymb,amsmath,latexsym,bm,stmaryrd}
\usepackage{spalign}
\RequirePackage[all,ps,cmtip]{xy}
\newdir^{ (}{{}*!/-5pt/\dir^{(}}
\newdir_{ (}{{}*!/-5pt/\dir_{(}}
\RequirePackage[mathscr]{eucal}
\usepackage{mathrsfs}\let\mathcal\mathscr
\theoremstyle{plain}

\newtheorem*{theorem*}{Main Theorem}
\newtheorem{theorem}{Theorem}
\newtheorem{proposition}[theorem]{Proposition}
\newtheorem{lemma}[theorem]{Lemma}
\newtheorem{corollary}[theorem]{Corollary}

\theoremstyle{definition}

\AtBeginDocument{%
}

\newcounter{toto}
\def\thetoto{\arabic{toto}}
\let\oldmarginpar\marginpar
\def\marginpar#1{\refstepcounter{toto}\textsuperscript{\textup{[\thetoto]}}\oldmarginpar{\footnotesize\textsuperscript{[\thetoto]}\,#1}}

\def\mainmatter{\renewcommand{\baselinestretch}{1.1}\normalfont}

\makeatletter
\@namedef{subjclassname@2010}{\textup{2010} Mathematics Subject Classification}
\def\l@section{\@tocline{1}{0pt}{0pc}{}{}}
\def\l@subsection{\@tocline{2}{0pt}{1.5pc}{}{}}
\def\l@subsubsection{\@tocline{3}{0pt}{2pc}{}{}}
\makeatother
\makeatletter
\@namedef{subjclassname@2020}{%
   \textup{2020} Mathematics Subject Classification}
\makeatother

\newcommand{\C}{\mathbb{C}}\let\CC\C

\newcommand{\R}{\mathbb{R}}

\newcommand{\Z}{\mathbb{Z}}

\newcommand{\bD}{\boldsymbol{D}}
\newcommand{\shhom}{\mathcal{H}\!\mathit{om}}

\DeclareMathOperator{\rh}{\mathit{R}\shhom}\let\Rhom\rh
\DeclareMathOperator{\tho}{\mathit{T}\shhom}

\DeclareMathOperator{\Rh}{\mathrm{RHom}}
\DeclareMathOperator{\RH}{RH}

\newcommand{\rb}{\mathrm{b}}

\newcommand{\coh}{\mathrm{coh}}
\newcommand{\hol}{\mathrm{hol}}
\newcommand{\rhol}{\mathrm{rhol}}

\newcommand{\Mod}{\mathrm{Mod}}

\newcommand{\op}{\mathrm{op}}

\newcommand{\sa}{\mathrm{sa}}

\newcommand{\Rc}{{\R\textup{-c}}}

\newcommand{\XS}{X\times S}

\newcommand{\XpS}{X'\times S}

\newcommand{\DXS}{\shd_{\XS/S}}

\newcommand{\DXpS}{\shd_{\XpS/S}}

\DeclareMathOperator{\Char}{Char}

\DeclareMathOperator{\rD}{\mathsf{D}}

\DeclareMathOperator{\Db}{\mathscr{D}b}

\let\Id\id

\DeclareMathOperator{\supp}{Supp}

\let\bar\overline

\let\epsilon\varepsilon

\let\setminus\smallsetminus
\let\leq\leqslant

\def\loccit{loc.\kern3pt cit.{}\xspace}
\def\cf{cf.\kern.3em}
\def\Cf{Cf.\kern.3em}
\def\eg{e.g.\kern.3em}

\def\resp{\text{resp.}\kern.3em}

\newcommand{\Df}{{}_{\scriptscriptstyle\mathrm{D}}f}

\newcommand{\Di}{{}_{\scriptscriptstyle\mathrm{D}}i}

\newcommand{\cbbullet}{{\raisebox{1pt}{$\sbullet$}}}
\newcommand{\sbullet}{{\scriptscriptstyle\bullet}}
\newcommand{\pOS}{p^{-1}\sho_S}

\def\shd{\mathcal{D}}

\let\cF F
\let\cG G
\def\shh{\mathcal{H}}

\def\shj{\mathcal{J}}

\def\shm{\mathcal{M}}

\def\sho{\mathcal{O}}
\def\shr{\mathcal{R}}

\newcommand{\RedefinitSymbole}[1]{%
\expandafter\let\csname old\string#1\endcsname=#1
\let#1=\relax
\newcommand{#1}{\csname old\string#1\endcsname\,}%
}
\RedefinitSymbole{\forall} \RedefinitSymbole{\exists}

\def\to{\mathchoice{\longrightarrow}{\rightarrow}{\rightarrow}{\rightarrow}}

\def\To#1{\mathchoice{\xrightarrow{\textstyle\kern4pt#1\kern3pt}}{\stackrel{#1}{\longrightarrow}}{}{}}

\let\oldbigoplus\bigoplus
\renewcommand{\bigoplus}{\mathop{\textstyle\oldbigoplus}\displaylimits}
\let\oldbigwedge\bigwedge
\renewcommand{\bigwedge}{\mathop{\textstyle\oldbigwedge}\displaylimits}
\let\oldbigcap\bigcap
\renewcommand{\bigcap}{\mathop{\textstyle\oldbigcap}\displaylimits}
\let\oldbigcup\bigcup
\renewcommand{\bigcup}{\mathop{\textstyle\oldbigcup}\displaylimits}

\DeclareMathOperator{\Supp}{Supp}

\begin{document}

\author[T. Monteiro Fernandes]{Teresa Monteiro Fernandes}
\address[T. Monteiro Fernandes]{Centro de Estudos Matem\'aticos and Departamento de Matem\' atica da Faculdade de Ci\^en\-cias da Universidade de Lisboa, Bloco C6, Piso 2, Campo Grande, 1749-016, Lisboa
Portugal}
\email{mtfernandes@ciencias.ulisboa.pt}
\thanks{The research of T. Monteiro Fernandes was supported by
CEMS.UL-Center for Mathematical Studies, Ciências, ULisboa
UID/04561/2025-https://doi.org/10.54499/UID/04561/2025}
\title{On the solutions of relative regular holonomic $\shd$-modules}

\date{}

\keywords{Relative $\mathcal D$-module, regular holonomic $\mathcal D$-module, relative constructible sheaf}

\subjclass[2020]{14F10, 32C38, 35A27, 58J15}

\begin{abstract}
Let $S$ be a complex curve and let $X$ be a complex manifold. Given a relative regular holonomic $\DXS$-module $\shm$ with respect to the projection $X\times S\to S$, we construct an isomorphism between its solutions respectively on the (relative subanalytic) sheaf of relative tempered holomorphic functions and on the sheaf of holomorphic solutions on $X\times S$.  
\end{abstract}

\maketitle
\tableofcontents
\mainmatter

\vspace*{-2\baselineskip}\vskip0pt%
\section{Introduction}
The relative framework we deal with in this note is associated to a projection $$p: X\times S\to S$$ where $X$ and $S$ are complex manifolds. 
Throughout this work we identify the relative cotangent bundle $T^*(X\times S/S)$ to $T^*X\times S$ and $d_X$ and $d_S$ will denote respectively the complex dimension of $X$ and of $S$. We denote by $p_X$ the projection $X\times S\to X$. 

Let $\DXS$ be the subsheaf of $\shd_{\XS}$ of operators commuting with $p^{-1}\sho_S$ and let $\Mod_{\coh}(\DXS)$ be the abelian category of coherent $\DXS$-modules.
A $\DXS$ -holonomic module is a coherent $\DXS$-module whose characteristic variety is contained in a product $\Lambda\times S$ where $\Lambda$ is $\C^*$-conic analytic Lagrangian in $T^*X$ (see \cite{MFCS1} for instance). The datum of a strict (i.e, a $p^{-1}\sho_S$-flat) holonomic $\DXS$-module is equivalent to the datum of a flat family of holonomic $\shd_X$-modules with characteristic variety contained in $\Lambda$. 
We say that a holonomic $\DXS$-module $\shm$ is regular holonomic if that is so for the restrictions of $\shm$ to each fiber of $p$ regarded as holonomic $\shd_X$-modules. 

 Let us give a brief historical motivation for the present work:
 
  M. Kashiwara and P. Schapira  in \cite{KS6} and \cite{KS7} introduced and studied the  subanalytic sheaf $\sho^t_X$ of tempered holomorphic functions and proved that, if $\shm$ is a regular holonomic $\shd_X$-module, then (with some abuse of notation),  the natural morphism 
 \begin{equation}\label{E0}
 \Rhom_{\shd_X}(\shm, \sho^t_X)\to \Rhom_{\shd_X}(\shm, \sho_X)
 \end{equation}
is an isomorphism. Note that $\sho^t_X$ is a complex (it is an ind-sheaf) and it encodes a wide family of sheaves. Examples are $\sho^{t, U}_X:=\Rhom(\C_U, \sho^t_X)$ the sheaf of holomorphic functions on $U$ which extend as distributions to $X$, where $U\subset X$ is a relatively compact Stein open subanalytic set. Another instance is the sheaf of distributions $\Db_M$ on a real analytic manifold $M$, if $X$ is a complexification of $M$, which is isomorphic, up to a shift, to $\Rhom(\bD'_X(\C_M),\sho^t_X)$. More generally, if $F$ is an $\R$-constructible complex on $X$, one can define a subanalytic sheaf $\Rhom(F, \sho^t_X)$ and, if one restricts $F$ to $\C$-constructible objects (i.e bounded complexes with $\C$-constructible cohomology), 
$\Rhom(F, \sho^t_X)$ is the essential tool for the Riemann-Hilbert-Kashiwara reconstruction functor.
The isomorphism \eqref{E0} is a consequence of Kashiwara's result \cite[Cor. 8.6]{Ka3} which asserts that, for any regular holonomic $\shd_X$-module $\shm$ and  for any $\R$-constructible complex $F$ on $X$, denoting by $\RH$ Kahsiwara's Riemann-Hilbert functor (\cf \cite{Ka3}, where $\RH$ is denoted by $\Psi$), the natural morphism 
\begin{equation}\label{Intro}\Rhom_{\shd_X}(\shm, \RH(F))\to \Rhom_{\shd_X}(\shm, \Rhom(F, \sho_{X}))
\end{equation} is an isomorphism.

Inspired by the study of subanalytic sheaves  developed by Kashiwara and Schapira, the relative subanalytic sheaf of relative tempered holomorphic functions $\sho^{t,S}_{\XS}$ was introduced in \cite{MFP1} and \cite{MFP2}. For instance, if $F$ is a relative strict perverse object, $\Rhom_{p^{-1}\sho_S}(F,\sho^{t,S}_{\XS})$ encodes the notion of holomorphic family of regular holonomic $\shd_X$-modules whose family of solutions complexes is $F$. 

We recall that $\sho^{t,S}_{\XS}$ is one of the key tools for the relative Regular Riemann-Hilbert correspondence  proved in \cite{FMFS1} (the case $d_S=1$) and \cite{FMFS2} (the general case) since it allows to construct the relative reconstruction Riemann-Hilbert functor $\RH^S$.

The natural next step was then to prove that, if $\shm$ is a relative regular holonomic complex, for any $F\in\rD^\rb_{\Rc}(\pOS)$ the natural morphism
\begin{equation}\label{E01}
 \Rhom_{\DXS}(\shm, \RH^S(F))\to \Rhom_{\DXS}(\shm, \Rhom(F, \sho_{X\times S})[d_X])
 \end{equation}
is an isomorphism. We prove it when $d_S=1$ which is our main result Theorem \ref{Thmain}. The main idea of the proof is to reduce to the case where $\shm$ is a single module and to combine induction on the dimension of $\Supp_X\shm:=p_X(\Supp\shm)$ with treating in separate the $\sho_S$-torsion case and the $\sho_S$-torsion free case. The latter, because $d_S=1$, amounts to assume that $\shm$ is a coherent locally free $\sho_{X\times S}$-module.

It is of course natural to ask why not consider, as a next step, a general $d_S$. And the reason is that, unlike the case of $S$-$\C$-constructibility treated in \cite{FMFS1}, \cite{FMFS2}, when $d_S>1$, $S$-$\R$ constructibility offers further (and not overcome) difficulty due to the lack of functorial properties of the functor $\RH^S$ when defined in the derived category of complexes with $S$-$\R$-constructible cohomology, namely with respect to base  pull-back and projective base push-forward. As an instance of this we refer \cite [Th.1.3]{FMF2}.

Such obstruction prevented us from following the strategy of the proof of Theorem 2 in \cite{FMFS2} in order to obtain the statement of Theorem \ref{Thmain} for a parameter manifold $S$ of arbitrary dimension. However the question remains open up to discover a new strategy.

We thank Luisa Fiorot for the patient revision of a previous version.

\section{A short reminder on the relative Riemann-Hilbert correspondence}

Below we summarize the background from \cite{MFCS1}, \cite{MFCS2}, \cite{FMFS1}, \cite{FMFS2} that we shall need in the sequel. We shall keep the notation $p$ for the projection in the parameter space $X\times S\to S$ avoiding reference to $X$ or to $S$.

\subsection{Holonomic and regular holonomic  $\DXS$-modules}
\begin{enumerate}
\item{We say that a $\pOS$-module is strict if it is flat over $\pOS$.}
\item{We recall that $\shm\in\Mod_{\coh}(\shd_{X\times S/S})$ is holonomic if the characteristic variety $\Char(\shm)$ is contained in $\Lambda\times S$, where $\Lambda$ is analytic $\C^*$-conic Lagrangian subset of $T^*X$; we denote by $\rD^{\rb}_{\hol}(\DXS)$ the associated triangulated category whose objects are the bounded complexes with holonomic cohomologies.}

\item{There is a well defined duality functor $$\bD: \rD^{\rb}_{\hol}(\DXS)\to \rD^{\rb}_{\hol}(\DXS)^\op$$ given by
$$\bD \shm:=\Rhom_{\DXS}(\shm, \DXS\otimes _{\pOS}\Omega_{\XS/S}^{\otimes^{-1}})[d_X])$$ where $\Omega_{\XS/S}$ denotes the sheaf of relative differential forms of maximal degree.}
\item{$\bD$ is an involution, i.e. $\bD\bD=\Id$.}
\item{We recall a tool introduced in \cite{MFCS1}, the holomorphic restriction to each fiber of $p$:  $$\forall s\in S,\, Li^*_s(\cbbullet):=\cbbullet\overset{L}{\otimes}_{p^{-1}\sho_S}p^{-1}(\sho_S/\shj_s)$$ where $\shj_s$ is the maximal ideal of functions vanishing in $s$.}
\item{ A Nakayama's Lemma variation: Let $\shm\in\rD^\rb_\hol(\DXS)$ and assume that $L i^*_{s_o}\shm=0$ for each $s_o\in S$. Then $\shm=0$.}
\item{Let $\shm$ be an object of $\rD^{\rb}_{\hol}(\DXS)$. Then $\bD \shm$ is concentrated in degree zero and $\shh^0\bD \shm$ is strict if and only if $\shm$ is itself concentrated in degree zero and $\shh^0\shm$ is a strict $\DXS$-module.}
\item{We say that $\shm\in\Mod(\shd_{X\times S/S})$ is regular holonomic if it is holonomic and $\forall s\in S, Li^*_s\shm\in\rD^{\rb}_{\rhol}(\shd_X)$; we denote by $\rD^{\rb}_{\rhol}(\DXS)$ the associated triangulated full subcategory of $\rD^\rb_{\hol}(\DXS)$.}
\item{$\rD^{\rb}_{\rhol}(\DXS)$ is stable by duality.}
\item{$\Mod_{\hol}(\DXS)$ and $\Mod_{\rhol}(\DXS)$ are closed under taking extensions in $\Mod(\DXS)$ and subquotients in $\Mod_{\coh}(\DXS)$.}
\end{enumerate}

\subsection{$S$-$\R$-constructibility}

\vspace{2mm}
Let $X$ be a real analytic manifold.
An $S$-locally constant coherent sheaf is a sheaf $L$ of $p^{-1}\sho_S$-modules such that, locally on $X\times S$, $L$ is isomorphic to $p^{-1}G$ where $G$ is an $\sho_S$-coherent module. Such an $L$ is also called $S$-local system. We recall the following full triangulated subcategory of $\rD^\rb(p^{-1}\sho_S)$.

\begin{itemize}
\item{An object $F\in\rD^\rb(p^{-1}\sho_S)$ is an object of $\rD^\rb_\Rc(\pOS)$ if there exists a subanalytic stratification $(X_{\alpha})_{\alpha\in A}$ of $X$, such that $\forall j\in\Z, \forall \alpha\in A, \shh^jF|_{X_{\alpha}\times S}$ is $S$-locally constant coherent. We also say that $F$ is $S$-$\R$-constructible.} 
\item{If $F\in\rD^\rb_\Rc(\pOS)$ then for each $x\in X$, $F|_{\{x\}\times S}$ belongs to $\rD^\rb_{\coh}(\sho_S)$.}
\item{There is a natural duality functor $\bD: \rD^\rb_\Rc(\pOS)\to \rD^\rb_\Rc(\pOS)^\op$ which is an involution given by
$$\bD F=\Rhom_{\pOS}(F,\pOS)[2 d_X]$$}
\end{itemize}
The following result will not be directly used in this note. We choose to include it for completeness.
\begin{proposition}\label{P1}
Let $F$ be an object of $\rD^\rb_{\Rc}(\pOS)$. Then, the support of $F$ satisfies: 
\begin{equation}\label{E002}
\supp F=\bigcup\limits_{i\in I}X_i\times S_i
\end{equation}
for some compact subanalytic contractible subsets $X_i$ of $X$ and some closed analytic subsets $S_i$ of $S$, and, locally on $X$, the set $I$ is finite.
Moreover, the sets $S_i$ are independent of the family $(X_i)_{i\in I}$.
\end{proposition}
\begin{proof}

By the assumption of $S$-$\R$-constructibility and the triangulation theorem (\cite[Proposition 8.2.5]{KS1}), we can consider a locally finite covering of $X$, $X=\cup_{i\in I} X_i$, where each $X_i$ is compact, subanalytic and contractible, and such that $F|_{X_i\times S}$ is isomorphic to $p^{-1}_{X_i}G_i$ where each $G_i$ is an object of 
$\rD^\rb_{\coh}(\sho_S)$. More precisely, given $x\in X$ and any such $X_i$ such that $x\in X_i$, $G_i$ is isomorphic to $F|_{\{x\}\times S}$ thus each $G_i$
 is unique up to isomorphism. Thus, for each $i\in I$
 $$\supp(F|_{X_i\times S})=X_i\times S_i$$ where $S_i=\supp G_i$ (which of course may be empty) thus $S_i$ is closed analytic in $S$.
 
We finally have $$\supp F=\bigcup\limits_{i\in I}\supp F|_{X_i\times S}=\bigcup\limits_{i\in I}X_i\times S_i$$
\end{proof}

We denote by $\supp_SF$ the set $\bigcup\limits_{i\in I}S_i$ which can fail to be analytic if $I$ is not finite. However we can define like in \cite[(2,7)]{FMFS2} $\dim \supp_SF:=\sup \dim S_i\leq d_S$.

\subsection{The functor $\RH^S$}With the subanalytic tools developed in \cite{MFP1}, \cite{MFP2}, the relative functor $\RH^S$ was first introduced in \cite{MFCS2}, followed by \cite{FMFS1} (case $d_S=1$) and by  \cite{FMFS2} (general case). Kashiwara's functor $\RH$ (\cf \cite{Ka3}) is recovered with $d_S=0$. 
 Below we give a short reminder of its construction and main results:

 Let $\rho_S: X\times S\to X_{sa}\times S$ be the natural morphism of sites introduced in \cite{MFP2}. The functor $\rho_S^{-1}$ admits a left adjoint $\rho_{S!}$ which is exact. 
 
 For $\shr=\DXS, \sho_{\XS}, \pOS$, $\rho_{S!}$ is an exact functor $\Mod(\shr)\to\Mod(\rho_{S!}\shr)$ which is a left adjoint of $\rho_S^{-1}: \Mod(\rho_{S!}\shr)\to \Mod(\shr)$.
 
 We note $\sho_{X\times S}^{t, S}$ the relative subanalytic sheaf on $X_{sa}\times S$ associated in \cite{MFP2} to the subanalytic sheaf $\sho^t_{X\times S}$ on $(X\times S)_{sa}$ (introduced in \cite{KS6}, see also \cite{LP08}).
 
  Then there exists by adjunction a natural morphism in $\rD^b(\rho_{S!}\DXS)$
 \begin{equation}\label{E:t} 
 \sho^{t,S}_{\XS}\to R\rho_{S*}\sho_{\XS}
\end{equation}
 
 The functor $\RH^S$ on $\rD^b(\pOS)^\op$ is given by
 
 $$\RH^S(\cbbullet):=\rho_S^{-1}\Rhom_{\rho_{S*}\pOS}(R\rho_{S*}(\cbbullet), \sho^{t,S}_{X\times S})[d_X]$$
 
\subsection{Functorial properties}

When $f:Y\to X$ is a morphism we keep the notion $f$ for the morphism $f\times \Id_S$ and if $\pi:S'\to S$ is a morphism of the parameter spaces, we also keep the notation $\pi$ for the morphism $\Id_X\times \pi$.

We first recall the following result proved in \cite[Th. 5.11]{FMFS2}: 
\begin{theorem}\label{Theorem 5.11} Let $f : X\to Y$ be a morphism of complex analytic manifolds, let $F\in\rD^\rb_{\Rc}(\pOS)$
and assume that f is proper on $\Supp F$. Then there is a canonical isomorphism in
$\rD^\rb(\DXS)$ which is compatible in a natural way with the composition of morphisms
$$\Df_* \RH^S(F)
\simeq \RH^S(R f_*F ).$$ 
\end{theorem}

The following result is an adaptation of the proof of (5.16) of \cite{KS4} using Theorem 5.11 of \cite{FMFS2}:
\begin{lemma}\label{L1}
Let $i: Y\to X$ be the embedding of a closed submanifold. Then we have an isomorphism of functors on $\rD^\rb_{\R-c}(\pOS)$:
$$\RH^S(i^{-1}F)[-d_Y]\overset{\simeq}{\to}\Di^*\RH^S(F)[-d_X]$$
\end{lemma}

We recall that $p^{-1}$ commutes with $R\pi_*$ as functors on $\rD^\rb(\sho_S)$.

 \section{Main result}
 The following result is the relative variant of \cite[Cor. 8.6]{Ka3}:
 \begin{theorem}\label{Thmain}
 Let us assume that $d_S=1$. Let $\shm\in\rD^\rb_{\rhol}(\DXS)$ and let $F\in \rD^\rb_{\Rc}(\pOS)$. Then we have a natural isomorphism:
 \begin{equation}\label{E1}\Rhom_{\DXS}(\shm, \RH^S(F)[-d_X])\to \Rhom_{\DXS}(\shm, \Rhom_{\pOS}(F, \sho_{X\times S}))
 \end{equation}
 \end{theorem}
 \begin{corollary}\label{Cormain}
 
Let $\shm\in\rD^\rb_{\rhol}(\DXS)$. Then the natural morphism
\begin{equation}\label{E1}\Rhom_{\rho_{S!}\DXS}(\rho_{S!}\shm, \sho^{t,S}_{\XS})\to 
\Rhom_{\rho_{S!}\DXS}(\rho_{S!}\shm, R\rho_{S*}\sho_{\XS})
\end{equation}
is an isomorphism.
\end{corollary}

\begin{proof}
Since the family of open sets of the form $U\times V$, with $U$ subanalytic relatively compact in X, generates the open coverings of $X_{\sa}\times S$, it is sufficient to prove that, for any open subanalytic relatively compact set $U$ in X and any open suset ~$V$ in $S$, morphism \eqref{E:t} induces an isomorphism,
$$R\Gamma(U\times V,\rh_{{\rho_S}_!\DXS}({\rho_S}_!\shm, \sho_{\XS}^{t,S}))$$ $$ \to R\Gamma(U\times V;\rh_{{\rho_S}_!\DXS}({\rho_S}_!\shm,R{\rho_S}_* \sho_{\XS}))$$

We have a chain of isomorphisms

$$R\Gamma(U\times V,\rh_{{\rho_S}_!\DXS}({\rho_S}_!\shm, \sho_{\XS}^{t,S}))$$ 
$$\simeq \Rh({\rho_S}_*\CC_{U\times V},\rh_{{\rho_S}_!\DXS}({\rho_S}_!\shm, \sho_{\XS}^{t,S}))$$
$$\simeq \Rh_{{\rho_S}_!\DXS}({\rho_S}_!\shm, \rh(\CC_{U\times V}, \sho_{\XS}^{t,S}))$$
$$\simeq \Rh_{\DXS}(\shm, {\rho_S}^{-1}\rh({\rho_S}_*\CC_{U\times V},\sho_{\XS}^{t,S}))\quad\text{(by adjunction)}$$
$$\simeq \Rh_{\DXS}(\shm, \rh(\CC_{X\times V}, \tho(\CC_{U \times S},\sho_{\XS})))$$
$$\simeq \Rh(\CC_{X\times V}, \rh_{\DXS}(\shm, \tho(\CC_{U \times S},\sho_{\XS})))$$

Similarly we have the chain of isomorphisms:
\begin{align*}
R\Gamma(U\times V&;\rh_{{\rho_S}_!\DXS}({\rho_S}_!\shm,R{\rho_S}_* \sho_{\XS}))\\
&\simeq \Rh({\rho_S}_*\C_{U\times V},\rh_{{\rho_S}_!\DXS}({\rho_S}_!\shm,R{\rho_S}_*\sho_{\XS}))\\
&\simeq \Rh_{{\rho_S}_!\DXS}({\rho_S}_!\shm,\rh({\rho_S}_*\C_{U\times V},R{\rho_S}_*\sho_{\XS}))\\
&\simeq \Rh_{\DXS}(\shm,{\rho_S}^{-1}\rh({\rho_S}_*\C_{U\times V},R{\rho_S}_*\sho_{\XS}))\\
&\simeq \Rh_{\DXS}(\shm,\rh(\C_{X\times V},\Rhom(\C_{U \times S},\sho_{\XS})))\\
&\simeq\Rh(\C_{X\times V},\rh_{\DXS}(\shm,\rh(\C_{U \times S},\sho_{\XS}))).
\end{align*}

 We have thus reduced the proof to showing that the morphism
\begin{multline*}
\rh_{\DXS}(\shm, \tho(\CC_{U\times S},\sho_{\XS}))\simeq \rh_{\DXS}(\shm, \RH^S(\CC_{U\times S})[-d_X])\\
\to \rh_{\DXS}(\shm, \rh(\CC_{U\times S},\sho_{\XS}))
\end{multline*}
is an isomorphism in $\rD^ \rb(\pOS)$ which follows from Theorem \ref{Thmain}.\end{proof}

\section{Proof of Theorem \ref{Thmain}}

In this proof we shall use the equivalence between left and right $\DXS$-modules, thus viewing a right $\DXS$-module as a left $\DXS^{op}$-module (\cite[1.-4.]{Ka2})

We shall also consider the following statements which are equivalent to the statement of Theorem \ref{Thmain}: 
\begin{enumerate}
\item{Let $\shm\in\rD^\rb_{\rhol}(\DXS^\op)$ and let $F\in \rD^\rb_{\Rc}(\pOS)$. Then we have a natural isomorphism:
$$\shm\otimes^L_{\DXS} \RH^S(F)[-d_X]\overset{\sim}{\to}\shm\otimes^L_{\DXS}\Rhom_{\pOS}(F, \sho_{X\times S})$$
}
 \item{Let $\shm\in\rD^\rb_{\rhol}(\DXS^\op)$ and let $F\in \rD^\rb_{\Rc}(\pOS)$. Then 
$$ \shm\otimes^L_{\DXS} \RH^S(F)$$
 is an object of $\rD^\rb_{\Rc}(\pOS)$.}
 \item{Let $\shm\in\rD^\rb_{\rhol}(\DXS)$ and let $F\in \rD^\rb_{\Rc}(\pOS)$. Then 
$ \Rhom_{\DXS}(\shm, \RH^S(F))$
 is an object of $\rD^\rb_{\Rc}(\pOS)$.}
\end{enumerate}
We start by noticing that, due to the holonomicity of $\shm$, for $F\in \rD^\rb_{\Rc}(\pOS)$ $$\Rhom_{\DXS}(\shm, \Rhom_{\pOS}(F, \sho_{X\times S}))$$ is an object of $\rD^\rb_{\Rc}(\pOS)$.

The equivalences of $(b)$ with $(c)$ and that of Theorem \ref{Thmain} with (a) are trivial. The equivalence of $(a)$ with  $(b)$ follows from \cite[Prop. 1.3]{MFCS2}.

Let $d_{\shm}$ denote the dimension of $\supp_X \shm:=p_X(\supp \shm)$.

1) \textit{We first prove that (c) holds in the case $d_{\shm}=0$ and arbitrary $d_S$}. 
For that purpose we may assume that $\supp_X \shm=\{0\}\subset X$. Then, by Kashiwara's equivalence, $\shm=\Di_*\shm_0$ where $i:\{0\}\times S\to X\times S$ and $\shm_0$ is identified to a coherent $\sho_S$-module.

Therefore, by the adjunction formula \cite[Th. 4.33 (2)]{Ka2} and Lemma \ref{L1}, we have $$\Rhom_{\DXS}(\Di_*\shm_0, \RH^S(F)[-d_X])\simeq Ri_*\Rhom_{\sho_S}(\shm_0, \Di^*\RH^S(F)[-d_X])$$ $$\simeq Ri_*\Rhom_{\sho_S}(\shm_0, \RH^S(i^{-1}F))$$ $$\simeq Ri_*\Rhom_{\sho_S}(\shm_0,\Rhom_{\sho_S}(i^{-1}F, \sho_S))$$ Then (c) holds for $\shm$ because $\Rhom_{\sho_S}(\shm_0, \Rhom_{\sho_S}(i^{-1}F, \sho_S))$ has $\sho_S$-coherent cohomologies.

 2) \textit{We now  prove (b) by induction on $d_{\shm}$}.
 
  Note that the result is of local nature since the morphism is well defined. We set $Z=\supp_X \shm$. 

2.1) \textit{The torsion case:}
By a standard method we may assume also that $\shm$ is a single module and then the result follows for arbitrary $d_{\shm}$ by Theorem 3.1 and the proof of Theorem 3.2  of \cite{FMFS1}.

2.2) \textit{We now assume that (b) holds true for any $\shm'$ such that $d_{\shm'}<k-1$}. 

Let $\shm$ be such that $d_{\shm}=k$.

By a standard argument, we again reduce the proof to the case where $\shm$ is a single module and then to the case where $\shm$ is a strict right regular holonomic $\DXS$-module.

Since the proof is of local nature, we can use the following classical argument:
Let us write $Z=\bigcup\limits_{i\in K} Z_i$ where each $Z_i$ is closed analytic irreducible as described in \cite[Lem. 2.10]{FMF}.
Let $Y\subset X$ be a hypersurface transversal to $Z$ containing the singular locus $Z_{sing}$ of $Z$ and all $Z_i$ such that
$\dim Z_i < k$. 

We can then  find a projective morphism $f:X'\to X$ with $\dim X'=k$, which is biholomorphic from the complement $X'\setminus Y'$ of a normal crossing divisor $Y'$ in $X'$ to the smooth locus of dimension $k$ of $Z$.

By construction $Z\cap Y$ has dimension strictly smaller than $k$. 
Let us consider the distinguished triangle in $\rD^\rb_{\Rc}(p^{-1}\sho_S)$:
$$F_{(X\setminus Z)\times S}\to F\to F_{Z\times S}\overset {+1}{\to}$$
We have $$\RH^S(F_{(X\setminus Z)\times S})\simeq \RH^S(F)[X\times S\setminus Z\times S]$$
By the regular  holonomicity of $\shm$ we may apply Nakayama's Lemma to conclude $R\Gamma_{[Z\times S]}(\shm)\simeq \shm$ by \cite[Prop. 3.32]{Ka2}, thus, by the relative version of \cite[Prop. 3.34]{Ka2} we have $$\shm\otimes^L_{\DXS}\RH^S(F_{(X\setminus Z)\times S})=0$$
Therefore we may assume that $\supp F\subset Z\times S$.
On the other hand, by the induction hypothesis, (b) holds true for the complex $R\Gamma_{[Y\times S]}\shm$ since $\dim (Z\cap Y)<d_{\shm}$, so we are reduced to assume that $\shm\simeq \shm(\ast Y\times S)$. As a consequence we may also assume that $\RH^S(F)\simeq \RH^S(F)(\ast Y\times S)$. 
Since 
$$\RH^S(F)(\ast Y\times S)\simeq \RH^S(F\otimes \C_{(X\setminus Y )\times S})$$ it follows that we may also assume $F\simeq F\otimes \C_{(X\setminus Y )\times S}$.  In other words, we may assume that 
$f^{-1}F = f^{-1}F\otimes \C_{(X' \setminus Y' )\times S}$ and $F\simeq Rf_*f^{-1}F$.  

As explained in \cite[Proof of Prop. 4. 27]{FMFS2}, setting $\delta:=k-\dim X=\dim X'-\dim X$, $\shm':=\Df^ *\shm[\delta] $ is concentrated in degree zero and is of D-type along $Y'$. 
According to Lemma 2.14 in \cite{FMFS1}, Theorem \ref{Thmain} holds true if $\shm$ is a left strict module of $D$-type along a divisor $Y\times S$ and $F\simeq F\otimes \C_{(X\setminus Y)\times S}$. (Remark that a slight adaptation of the proof of the loc.cit Lemma, which consists in using right instead of left $\DXS$-modules, entails that (a) and thus (b) hold true for any right strict $\DXS$-module of $D$-type along a divisor $Y\times S$ and for  any $F$). 

Thus (b) holds true for $\shm'$ and $F':=f^{-1}F$. Therefore we conclude that 
\begin{equation}\label{E100}Rf_*(\shm'\otimes^L_{\DXpS}\RH^S_{X'}(f^{-1}F)[-d_{X'}])
\end{equation}
is an $S$-$\R$-constructible complex.
Still using \cite[Proof of Prop. 4. 27]{FMFS2}, via the translation left/right $\DXS$-modules, we have $$\shm'\simeq f^{-1}\shm\otimes_{f^{-1}\DXS(\ast Y\times S)}\DXpS(\ast Y'\times S)$$ thus 
$$Rf_*(\shm'\otimes^L_{\DXpS}\RH^S_{X'}(f^{-1}F)[-d_{X'}])$$ $$\simeq Rf_*(f^{-1}\shm\otimes^L_{f^{-1}\DXS}\RH^S_{X'}(f^{-1}F)[-d_{X'}])$$ 

Therefore, by the stability of $S$-$\R$-constructibility under proper direct image and the projection formula, \eqref{E100} implies that \begin{equation}\label{E101}\shm\otimes^L_{\DXS}Rf_*\RH^S_{X}(f^{-1}F)[-d_{X'}]
\end{equation} 
is an $S$-$\R$-constructible complex.

According to \cite[Th. 5.11]{FMFS2}, we have a natural isomorphism $\Df_*\RH^S_{X'}(f^{-1}F)\simeq \RH_X^S(Rf_*f^{-1}F)\simeq \RH^S(F)$. 

On the other hand, again by \cite[Proof of Prop. 4. 27]{FMFS2}
we also have
$$\Df_*\RH^S_{X'}(f^{-1}F)\simeq Rf_*\RH^S_{X'}(f^{-1}F)$$ 

Composing with \eqref{E101} ends the proof of b).

\section{Application}

For a complex manifold $Z$, $Z_{\R}$ stands for the underlying real analytic manifold and we abusively note $\Db_Z$ instead of $\Db_{Z_{\R}}$. We also note $\bar{Z}$ the complex conjugate of $Z$ so that we regard $Z\times \bar{Z}$ as a complexification of $Z_{\R}$ by the diagonal embedding. 

If $M$ is a real analytic manifold of dimension $n$ and $X$ is a complexification of $M$, we note $\Db_{M\times S/S}$ (resp. $B_{M\times S/S}$) the sheaf of distributions on $M\times S_{\R}$ (resp. the sheaf of hyperfunctions on $M\times S_{\R}$) which are holomorphic on $S$. 
We note $i: M\to X$ the real analytic embedding, notation we keep for the embedding $M\times S\to X\times S$. We also note $\Delta_S$ the diagonal in $S\times\bar{S}$. In a similar way we note $j: M\times \Delta_S\to X\times S\times\bar{S}$ the diagonal embedding associated to $S$. We also consider the projections $r: X\times S\times \bar{S}\to X\times S $ and $p:X\times S\times \bar{S}\to \bar{S}$.

\begin{lemma}\label{Ldisthyp}
Once fixed an orientation on $M$, we have isomorphisms of sheaves of $\DXS$-modules:
\begin{enumerate}
\item{$\Db_{M\times S/S}\simeq \tho(\C_{M\times S}, \sho_{\XS})[d_X]$.}
\item{$B_{M\times S/S}\simeq \Rhom(\C_{M\times S}, \sho_{\XS})[d_X]$.}
\end{enumerate}
\end{lemma}
\begin{proof} 
Let us consider the following commutative square of functors:
\begin{equation}\label{E300}
\xymatrix{
X\times S\times \bar{S}\ar[r]^r & X\times S\\
M\times \Delta_S\ar[u]^j\ar[r]^{r'}_{\simeq} &M\times S\ar[u]^i\\
}
\end{equation}

Therefore we have $j^! \circ r^{!}=r'^{!}\circ i^!$ and thus $$j^! r^{-1}[-d_S]=r'^{-1}j^!$$
We start by proving (b).

We have $B_{M\times S}\simeq \Rhom(\C_{M\times S}, \sho_{X\times S\times \bar{S}})[d_X+2d_S])|_{M\times S}$

Hence, identifying $S$ and $\Delta_S$ by the first projection ($r'$), we have $$\Rhom_{p^{-1}\shd_{\bar{S}}}(p^{-1}\sho_{\bar{S}}, \Rhom(\C_{M\times S}, \sho_{X\times S\times \bar{S}})[d_X+2d_S])|_{M\times S}$$
$$\simeq \Rhom (\C_{M\times S}, \Rhom_{p^{-1}\shd_{\bar{S}}}(p^{-1}\sho_{\bar{S}},  \sho_{X\times S\times \bar{S}})[d_X+2d_S])|_{M\times S}$$
$$\simeq \Rhom (\C_{M\times S}, r^{-1}\sho_{X\times S}[-d_S][d_X+d_S])|_{M\times S}$$
$$\simeq \Rhom(\C_{M\times S}, \sho_{X\times S})[d_X]$$ which ends the proof of $(b)$.

The proof of (a) is similar using the tools of \cite{KS4}.

According to \cite[Th. 5.10 (2.4)]{KS4} we have $$\Db_{M\times S}\simeq \tho(\C_{M\times S}, \sho_{X\times S\times \bar{S}}[d_X+2 d_S])|_{M\times S}$$

Then, in the argument above, we can replace $\tho$ by $\Rhom$, the sheaf $\sho_{X\times S\times \bar{S}}$ by the subanalytic sheaf $\sho^t_{X\times S\times \bar{S}}$ and the sheaf 
$\sho_{X\times S}$ by the subanalytic sheaf $\sho^t_{X\times S}$, pursuing the same argument because  we have $$\Rhom_{p^{-1}\shd_{\bar{S}}}(p^{-1}\sho_{\bar{S}}, \sho^t_{X\times S\times \bar{S}}))|_{M\times S}$$ $$\simeq \Rhom_{\shd_{X\times S\times\bar{S}}}(\shd_{X\times S\times\bar{S}\to \bar{S}}, \sho^t_{X\times S\times \bar{S}}))|_{M\times S}$$ $$\simeq r^{-1}\sho^t_{X\times S}[-d_S]$$ according to \cite[Lem. 7.4.8]{KS6}.

\end{proof}

Therefore, if $\shm$ is an object of $\rD^\rb_{\rhol}(\DXS)$ and $d_S=1$, Theorem \ref{Thmain} shows that we have an isomorphism 

$$ \Rhom_{\DXS}(\shm, \Db_{M\times S/S})\simeq \Rhom_{\DXS}(\shm,
B_{M\times S/S})$$

\end{document}